\documentclass{amsart}
\usepackage{amsmath,amssymb,epic,graphicx,mathrsfs,enumerate,xcolor}
\usepackage{amsfonts}
\usepackage{amsthm}
\usepackage{amssymb}
\usepackage{latexsym}
\usepackage{longtable}
\usepackage{epsfig}
\usepackage{amsmath}
\usepackage{hhline}
\usepackage{float}
\restylefloat{table}

\input xy
\xyoption{all}

\newtheorem{teor}{Theorem}

\newtheorem{prop}{Proposition}

\newtheorem{lemma}{Lemma}

\newcommand{\abs}[1]{\left\vert#1\right\vert}         

\newcommand{\PGammaL}{{\operatorname{P\Gamma L}}}

\begin{document}

\title[Maximal pairwise generating sets]{On the maximal number 
of elements pairwise generating the finite alternating group}

\author{Francesco Fumagalli} 
\address{Dipartimento di Matematica e Informatica 'Ulisse Dini', 
Viale Morgagni 67/A, 50134 Firenze, Italy}
\email{francesco.fumagalli@unifi.it}

\author{Martino Garonzi}
\address{Departamento de Matem\'atica, Universidade de Bras\'ilia, Campus 
Universit\'ario \\ Darcy Ribeiro, Bras\'ilia-DF, 70910-900, Brazil}
\email{mgaronzi@gmail.com}

\author{Pietro Gheri} 
\address{Dipartimento di Matematica e Informatica 'Ulisse Dini', 
Viale Morgagni 67/A, 50134 Firenze, Italy}
\email{pietro.gheri@unifi.it}

\thanks{The second author acknowledges the support of Funda\c{c}\~{a}o de Apoio \`a 
Pesquisa do Distrito Federal (FAPDF) - demanda espont\^{a}nea 03/2016, and of 
Conselho Nacional de Desenvolvimento Cient\'ifico e Tecnol\'ogico (CNPq) - Grant 
numbers 302134/2018-2, 422202/2018-5. The work of the second author on the project 
leading to this application has
received funding from the European Research Council (ERC) under the
European Union's Horizon 2020 research and innovation programme
(grant agreement No. 741420). He was also
supported by the National Research, Development and Innovation Office
(NKFIH) Grant No.~K132951, Grant No.~K115799, Grant No.~K138828.}

\date{}



\begin{abstract}
Let $G$ be the alternating group of degree $n$. 
Let $\omega(G)$ be the maximal size of a 
subset $S$ of $G$ such that $\langle x,y \rangle
= G$ whenever $x,y \in S$ and $x \neq y$ and 
let $\sigma(G)$ be the minimal size of a family of 
proper subgroups of $G$ whose union is $G$. We 
prove that, when $n$ varies in the family of composite 
numbers,
$\sigma(G)/\omega(G)$ tends to $1$ as 
$n \to \infty$. Moreover, we explicitly 
calculate $\sigma(A_n)$ for $n \geq 21$ 
congruent to $3$ modulo $18$.
\end{abstract}

\maketitle

\section{Introduction}
Given a finite group $G$ that can be 
generated by $2$ elements but not by $1$ 
element, set $\omega(G)$ to be the largest 
size of a \emph{pairwise generating set} 
$S \subseteq G$, that is, a subset $S$ of 
$G$ with the property that $\langle x,y 
\rangle = G$ for any two distinct elements 
$x,y$ of $S$. Also, set $\sigma(G)$ to be 
the \emph{covering number} of $G$, that is 
the minimal number of proper subgroups of $G$ 
whose union is $G$. We will reserve the 
term ``covering'' of $G$ for any family of 
proper subgroups of $G$ whose union is $G$.

Since any proper subgroup of $G$ contains 
at most one element of any pairwise 
generating set, $\omega(G) \leq \sigma(G)$ 
always.

S. Blackburn \cite{B} and L. Stringer 
\cite{stringer} proved that if $n$ is odd and 
$n \neq 9,15$ then $\sigma(S_n)=\omega(S_n)$ 
and that if $n \equiv 2 \pmod 4$ then $\sigma 
(A_n) =\omega (A_n)$. Stringer also proved 
that $\omega(S_9) < \sigma(S_9)$. In  
\cite{B} it is conjectured that, if $S$ is a finite non abelian simple group, then
$\sigma(S)/\omega(S)$ tends to $1$ as the 
order of 
$S$ goes 
to infinity. We remark that, apart from the 
above, the only cases in which the precise 
value of $\omega(G)$ is known are for groups
$G$ of 
Fitting height at most $2$ (\cite{LM}) and 
for certain linear groups (see \cite{BEGHM}).

In this paper, we prove the following results.

\begin{teor} \label{main}
Let $n$ vary in the set of composite 
positive integers. Then 
$$
\lim_{n \rightarrow \infty} \frac{\sigma 
(A_n) }{\omega (A_n)} =1.
$$
\end{teor}

Note that Stringer's result implies our 
theorem when $n \equiv 2 \pmod 4$, so we will 
only prove it when $n$ is 
divisible by $4$ or an odd composite integer.


Our second result concerns the precise value of  
$\sigma(A_n)$ when $n \equiv 3 \pmod{18}$.

\begin{teor} \label{alt3q}
Let $n > 3$ be an integer with $n \equiv 3 
\pmod{18}$ and let $q:=n/3$. Then
$$
\sigma(A_n) = \sum_{i=1}^{q-2} \binom{n}{i} 
+ \frac{1}{6} \frac{n!}{q!^3}.
$$
A minimal covering of $A_n$ consists of 
the intransitive maximal subgroups of type
$(S_i \times S_{n-i})\cap A_n$, for  
$i=1,\ldots,q-2$, and the imprimitive maximal 
subgroups with $3$ blocks, which 
are isomorphic to 
$(S_q \wr S_3)\cap A_n$.
\end{teor}

\section{Technical lemmas}

In this section we will collect some technical 
results we will need throughout the paper.

\begin{lemma}[Lemma 10.3.3 in 
\cite{stringer}]\label{stringer} 
Let $n$ be an odd integer which is the 
product of at least three primes 
(not necessarily distinct). Then if $n$ is 
sufficiently large we have 
$$
\frac{|S_{n/p} \wr S_p|}{|S_{n/m} \wr S_m|} 
\geq 2^{\sqrt{n}-3}
$$
where $m$ is any nontrivial proper divisor 
of $n$ different from $p$.
\end{lemma}

The following lemma is a generalization of 
\cite[Lemma 4]{B} and its proof uses the 
same ideas. We need to apply it for 
$k \leq 3$. 

\begin{lemma} \label{ncube}
Let $a_1,\ldots,a_k$ be positive integers with 
$\sum_{i=1}^k a_i = n$. Let $M$ be a subgroup of $S_n$. 
The number of conjugates of $M$ in $S_n$ containing a fixed 
element of type $(a_1,\ldots,a_k)$ is at most $n^k$.
\end{lemma}

\begin{proof}
Let $N$ be the number of elements of $S_n$ of type 
$(a_1,\ldots,a_k)$ and let $g$ be an element 
of this type. 
We want to show that $N \geq n!/n^k$. 
Assume that $a_1,\ldots,a_k$ are organized so that $k_i$ of 
them equal $a_i$ for $i=1,\ldots,t$, so that 
$k_1 + \ldots + k_t = k$ and $a_1k_1 + \ldots + a_tk_t = n$. 
Since the centralizer of $g$ in $S_n$ is isomorphic to 
$\prod_{i=1}^t C_{a_i} \wr S_{k_i}$ we deduce that
\begin{align*}
N = \frac{n!}{a_1^{k_1} k_1! \cdots a_t^{k_t} k_t!} \geq 
\frac{n!}{(a_1k_1)^{k_1} \cdots (a_tk_t)^{k_t}} \geq 
\frac{n!}{n^k},
\end{align*}
being $n^k = (a_1k_1+\ldots+a_tk_t)^{k_1+\ldots+k_t}$.

Let us double-count the size of the set $X$ of pairs $(h,H)$ 
such that $H$ is a subgroup of $S_n$ conjugate to $M$ and 
$h \in H$ is of type $(a_1,\ldots,a_k)$. We have 
$|X| = N \cdot a(M)$ where $a(M)$ is the 
number of conjugates of $M$ containing a fixed $h \in G$ 
of this type 
and 
$|X|=|S_n:N_{S_n}(M)| \cdot b(M) \leq |S_n|$ where $b(M) 
\leq |M|$ is the number of elements of $M$ of type 
$(a_1,\ldots,a_k)$. We 
obtain $$n! \geq |X| = N \cdot a(M) \geq (n!/n^k) a(M).$$ 
It follows that $a(M) \leq n^k$.
\end{proof}

In the following we make frequent use of 
Stirling's inequalities, which holds for 
every integer $t\geq 2$,
\begin{equation}\label{eq_Stirling}
\sqrt{2 \pi t} \cdot (t/e)^t <t!< 
e\sqrt{t}\cdot (t/e)^t.    
\end{equation}

\begin{lemma}\label{l:y!/x!}
Let $X := \{(x,y) \in \mathbb{N}^2\ :\ 2 
\leq x \leq y-2\}$. Then 
$$\lim_{\substack{|(x,y)| \to \infty \\ 
(x,y) \in X}} \frac{y!^{x-1}}{x!^{y-1}} = + 
\infty.$$
\end{lemma}

\begin{proof}
Set $$f(x,y) := \frac{y!^{x-1}}{x!^{y-1}}.$$
Note that $f(x,y) \leq f(x,y+1)$ whenever 
$(x,y) \in X$. This is because the claimed 
inequality is equivalent to $x! \leq 
(y+1)^{x-1}$ for $(x,y)\in X$, which 
is a consequence of the 
inequality $x! \leq x^{x-1}$, which can be 
easily proved by induction.

We are left to check that $f(x,x+2)$ tends 
to $+\infty$ when $x$ tends to $+\infty$. 
This can be proved directly using calculus 
techniques and again the fact that
$x!\leq x^{x-1}$.
\end{proof}

\begin{lemma} \label{bimpr}
Let $d \geq 2$, $k \geq 5$ be integers such 
that and $n = dk \geq 26$. 
Then $$|S_d \wr S_k| = d!^k k! \leq (n/5e)^n 
(5n)^{5/2} e \sqrt{n}.$$
\end{lemma}

\begin{proof}
Set $f(x) := (nx)^{x/2}/x^n$ for any real 
$x \geq 5$. The derivative of 
$f$ is $f'(x) = (1/(2x)) f(x) g(x)$ where 
$g(x) = x \log(nx)-2n+x$, so 
$f'(5) < 0$ being $n \geq 14$, and
$f$ is 
decreasing in $x=5$. Since 
$g'(x) = \log(nx)+2$ is positive (being 
$x \geq 5$) and the sign of $f'$ 
equals the sign of $g$, we deduce that in 
the interval $[5,n/5]$ we have 
$f(x) \leq \max \{f(5),f(n/5)\}$ 
and this equals $f(5)$ being $n \geq 25$. 
Therefore, using the bound 
$m! \leq (m/e)^m e \sqrt{m}$, if $k\leq n/5$ 
we obtain
$$ d!^k k!  \leq (d/e)^{dk} e^k d^{k/2} 
(k/e)^k e k^{1/2} = f(k)\cdot (n/e)^n e \sqrt{k} 
  \leq (n/5e)^n (5n)^{5/2} e \sqrt{n}.$$
The case $k > n/5$ corresponds to $d < 5$ 
and can be done case by case. For $d=4$ we 
need $n \geq 18$, for 
$d=3$ we need $n \geq 2$ and for $d=2$ we 
need $n \geq 26$.
\end{proof}

\begin{lemma} \label{impr_4_n3}
Let $n$ be an even positive integer, 
$r$ an odd divisor of $n$ such that 
$3\leq r \leq n/3$ and set 
$g(n,r)=(r!)^{n/r} \left( n/r \right) ! + 
r\cdot r! \cdot ((n/r)!)^r$.
Then there exists a positive constant $C$ 
such that $g(n,r)\leq C\big((n/3)!\big)^3$
for every $n$ large enough.
\end{lemma}
\begin{proof}
Let $n$ and $r$ as in the statement and 
note that since $n$ is even and $r$ is odd we 
have $3\leq r\leq n/4$. If $r=3$ then 
$g(n,r)=6^{n/3} (n/3)! + 18((n/3)!)^3<19((n/3)!)^3$
for large enough $n$. Using Stirling's inequalities (\ref{eq_Stirling}), it is easy to prove that $\frac{(n/4)!^4}{(n/3)!^3} \leq n \cdot (3/4)^n$ for large enough $n$, and this easily implies the result when $r=n/4$.

Assume therefore that $5\leq r\leq n/5$. 
We apply Lemma \ref{bimpr} and deduce that 
\begin{align*}
g(n,r)&\leq (r+1)\left(\frac{n}{5e}
\right)^n(5n)^{5/2} e \sqrt{n} \leq 
\left(\frac{n}{5e}\right)^n(5n)^{5/2}
n^{3/2}e.    
\end{align*}
By another application of 
(\ref{eq_Stirling}), we get that 
there exists a positive constant $D$ such 
that 
$$\frac{g(n,r)}{((n/3)!)^3}\leq 
(3/5)^nn^{5/2}\cdot D,$$
which tends to zero as $n$ tends to infinity.  
\end{proof}

\section{Proof of Theorem \ref{main}}

Let $\Gamma$ be an (undirected) graph. Recall that 
the degree of a vertex of $\Gamma$ is defined as the 
number of vertices of $\Gamma$ that are adjacent to it. 
Also, a set of vertices is called 
independent if no two of its elements are 
connected by an edge. We prove Theorem \ref{main} as an 
application of  following result due to P.E.~Haxell.

\begin{teor}[Theorem 2 in \cite{haxell}] 
\label{haxell}
Let $k$ be a positive integer, let 
$\Gamma$ be a 
graph of maximum degree 
at most $k$, and let 
$V(\Gamma)= V_1 \cup \dots \cup V_n$ be a 
partition of the vertex set of $\Gamma$. 
Suppose that $|V_i|\geq 2k$ for each $i$. 
Then $\Gamma$ has an independent set $\{v_1, 
\dots ,v_n \}$ where $v_i \in V_i$ for each 
$i$. 
\end{teor}

We will apply Theorem \ref{haxell} to prove Theorem \ref{main}, first for $n$ odd and composite, then for $n$ divisible by $4$. This will correspond to two different graphs.

\subsection{Case $n$ odd}

We first consider the case $n$ is an odd 
composite number. In this section, $p$ will always be the smallest prime divisor of $n$. For each proper nontrivial divisor $m$ of 
$n$, let $\mathcal{P}_m$ be the set of 
partitions of the set $\{ 1 , \dots , n \}$ 
into $m$ blocks each of cardinality 
$n/m$.
We want to find a maximal set of 
$n$-cycles in $A_n$ pairwise generating $A_n$
and in particular we will prove the following.
\begin{prop}\label{prop:omega_n_odd}
If $n$ is a sufficiently large odd composite number and $p$ denotes the smallest prime divisor of $n$, 
then $\omega(A_n)\geq 
\abs{\mathcal{P}_p}$.
\end{prop}
Since the imprimitive maximal subgroups of $A_n$ preserving a 
partition with $p$ blocks cover all the $n$-cycles in 
$A_n$, and since the elements of $A_n$ which are not 
$n$-cycles are covered by the maximal intransitive subgroups of 
type $(S_i \times S_{n-i}) \cap A_n$ with $1 \leq i \leq n/3$, we 
deduce from Proposition \ref{prop:omega_n_odd} that
$$|\mathcal{P}_p| \leq \omega(A_n) \leq \sigma(A_n) \leq 
|\mathcal{P}_p| + \sum_{i=1}^{\lfloor n/3 \rfloor} 
\binom{n}{i}.$$
This proves Theorem \ref{main} since the sum on the right-hand 
side is less than $2^n$, so it is asymptotically irrelevant 
compared to
$$|\mathcal{P}_p| = \frac{n!}{(n/p)!^p p!} > \frac{(n/e)^n 
e^p (n/p)^{p/2}}{(n/pe)^n (p/e)^p e \sqrt{p}} = \frac{p^n}{e 
\sqrt{p}} \cdot \left( e^2 \sqrt{n/p^3} \right)^p \geq 3^n$$
where the last inequality holds for sufficiently large $n$.

We are therefore reduced to prove Proposition \ref{prop:omega_n_odd}.

For every $\Delta \in \mathcal{P}_m$ 
let $C(\Delta)$ be the set of 
$n$-cycles $x \in A_n$ such that 
$\Delta$ is the set of orbits of the element 
$x^m$. In other words, $C(\Delta)$ is the 
set of $n$-cycles contained in the maximal 
imprimitive subgroup of $A_n$ whose block 
system is $\Delta$. Using the fact that every $n$-cycle belongs to a unique imprimitive maximal subgroup of $S_n$ with $m$ blocks, it is easy to see that  
$|C(\Delta)|=|S_{n/m} \wr S_m|/n$ using a 
double counting argument.
With a slight abuse of notation, for any 
maximal subgroup $H$ of $A_n$, 
we call $C(H)$ the set of $n$-cycles 
contained in $H$.

We define a graph $\Gamma$ whose 
vertex set is $V(\Gamma)$ and whose edge set is 
$E(\Gamma)$ in the following way.
$V(\Gamma)$ is the set of 
$n$-cycles of $A_n$ and, for distinct $x,y \in 
V(\Gamma)$, 
we say that  $\{ x,y \} \in E(\Gamma)$ if 
and only if $\langle x,y \rangle 
\neq A_n$ and the orbits of $x^p$ do not 
coincide with those of $y^p$, in 
other words there is no $\Delta \in 
\mathcal{P}_p$ such that 
both $x$ and $y$ belong to $C(\Delta)$.

Since the $n$-cycles are pairwise conjugate in 
$S_n$, the graph $\Gamma$ is vertex-transitive, 
so it is regular, in other words,  
every vertex has the same valency $k$. In order 
to prove Proposition \ref{prop:omega_n_odd}, 
it is enough to prove that $|C(\Delta)| \geq 2k$ 
for all $\Delta \in \mathcal{P}_p$, since then 
the result will follow from Theorem \ref{haxell} 
applied to the partition of the vertex-set of 
$\Gamma$ given by the $C(\Delta)$ with 
$\Delta \in \mathcal{P}_p$.

If $x$ is any vertex, then
$$
k \leq \sum_{H \in \mathcal{H}_x} \left| 
C(H) \right|
$$
where $\mathcal{H}_x$ is the set of maximal 
subgroups of $A_n$ containing $x$, except for 
the maximal imprimitive subgroup with $p$ 
blocks. Clearly, no intransitive subgroup 
contains $x$ so $\mathcal{H}_x$ is made 
of imprimitive and primitive subgroups. Let 
$\mathcal{H}_x^{imp}$ be the set 
of maximal imprimitive subgroups of $A_n$ 
containing $x$ whose number 
of blocks is not $p$, and let 
$\mathcal{H}_x^{prim}$ be the set of maximal 
primitive subgroups of $A_n$ containing $x$.
Then
$$
k \leq \sum_{H \in \mathcal{H}_x^{imp}} 
\left| C(H) \right| + \sum_{H \in 
\mathcal{H}_x^{prim}} \left| C(H) \right| .
$$

We bound the first term of the above sum. 
Let $\Delta_m(x)$ be the partition in 
$\mathcal{P}_m$ whose blocks are the orbits 
of the element $x^m$.
Since $n$ has at most $2\sqrt{n}$ positive divisors,
\begin{equation} \label{odd_impr_bound}
\sum_{H \in \mathcal{H}_x^{imp}} \left| 
C(H) \right| = \sum_{m \mid n, m \neq p} 
|C(\Delta_m(x))|
 \leq 2\sqrt{n} 
\max_{m \mid n, m \neq p} |S_{n/m} 
\wr S_m|/n
\end{equation}
where the second summation and the maximum is on all 
nontrivial proper divisors $m$ of $n$ that are 
different from $p$.

Note that, if $n \neq p^2$, the last term in 
(\ref{odd_impr_bound}) is at least $c^n$ for any given constant 
$c$, if $n$ is sufficiently large. This can be checked easily 
using $|S_{n/m} \wr S_m|=(n/m)!^m m!$ and Stirling 
inequalities (\ref{eq_Stirling}).

Lemma \ref{stringer} implies that the last 
term in (\ref{odd_impr_bound}) 
is asymptotically irrelevant compared to
$|C(\Delta)|$ for $\Delta \in 
\mathcal{P}_p$ when $n$ is the product of at least three primes.

We now turn to primitive subgroups.
When $n>23$ (and $n$ is not a prime) the 
primitive maximal subgroups of $A_n$ 
containing $n$-cycles are 
permutational isomorphic to 
$\PGammaL(m,s) \cap A_n$, where 
\begin{equation}\label{eq:q=s^m-1/s-1}
n=\frac{s^m-1}{s-1}    
\end{equation}
for some $m\geq 2$ and some prime power $s$, 
its action on points (or on hyperplanes) 
of a projective space of dimension $m$ 
over the field of $s=p^f$  elements 
(\cite[Theorem 3]{Jones2002}). 
Therefore in particular we have that 
if $H$ is such a subgroup of $A_n$ then
$$\abs{C(H)}<\abs{\PGammaL_m(s)}<2^{n-1}.$$
By Lemma \ref{ncube}  there are at 
most $n$ conjugates of $H$ containing 
a fixed $n$-cycle $x$. Moreover we show now 
that $S_n$ has at most $\log_{2}(n)$ 
conjugacy classes of such maximal primitive 
subgroups, and therefore this number is at 
most $2\log_{2}(n)$ for $A_n$. 
First observe that, being $n$ odd, if 
$(m,s)$ is a pair satisfying equation 
(\ref{eq:q=s^m-1/s-1}) then $s$ is the 
biggest power of $p$ dividing $n-1$. The 
possible choices for $s$ are at most
the number of primes dividing $n-1$, that is 
at most $\pi(n-1)$ which is trivially 
smaller than $\log_2(n)$. Once that $s$ 
is chosen there is at most only one 
possible value for $m$
such that $(m,s)$ satisfies 
(\ref{eq:q=s^m-1/s-1}). Thus the 
number of these pairs 
is bounded from above by $\log_{2}(n)$.
Now, for a fixed pair $(m,s)$ satisfying 
(\ref{eq:q=s^m-1/s-1})
the group $A_n$ contains at most 
two conjugacy classes of primitive maximal 
subgroups isomorphic to $\PGammaL_m(s)\cap 
A_n$ (it can be proved that the two 
actions of such a group respectively on the
projective points and on the projective 
hyperplanes are equivalent in $S_n$). 
Thus the number of conjugacy classes of 
proper primitive maximal subgroups 
containing an $n$-cycle is at most 
$ 2\log_{2}(n)$.
It follows that 
$$
\sum_{H \in \mathcal{H}_x^{prim}} \left| 
C(H) \right| \leq n\log_2(n) \cdot 2^{n} 
$$
and so it is easy to see that the last term 
in (\ref{odd_impr_bound}) 
is an upper bound also  for this sum 
(remember that we are considering $n\neq p^2$).

Assume that $n$ is a product of at least three primes, not necessarily distinct. If $\Delta \in \mathcal{P}_p$, then 
by Lemma \ref{stringer} we have 
$$\frac{|C(\Delta)|}{k}  \geq 
\frac{|S_{n/p} \wr S_p|}{2\sqrt{n} 
\max_{p \neq m \mid n} |S_{n/m} 
\wr S_m|} \geq 
\frac{2^{\sqrt{n}-4}}{\sqrt{n}} > 2,$$
which gives the result (here again 
the maximum is on all nontrivial proper divisors $m$ 
of $n$ that are different from $p$).

When $n=pq$ for some prime $q$ distinct from $p$, 
arguing as 
above we have $k \leq 2|S_p \wr S_q|$ for 
large enough $n$, therefore
\[
\frac{|C(\Delta)|}{k}  \geq \frac{|S_{q} \wr 
S_p|}{2|S_{p} \wr S_q|}
\]
and we only have to prove that, for large 
enough $n$, the right-hand side is larger 
than $2$. This follows from Lemma \ref{l:y!/x!}.

Assume finally that $n=p^2$. Then the only contribution is 
given by primitive subgroups, in other words
$$\frac{C(\Delta)}{k} \geq \frac{|S_p \wr S_p|/n}{\sum_{H \in \mathcal{H}_x^{prim}} |C(H)|} \geq \frac{p!^{p+1}/n}{n \log_2(n) \cdot 2^n}.$$
This is clearly larger than $2$ for large enough $n$.

This concludes the proof of Proposition \ref{prop:omega_n_odd} and therefore also of Theorem \ref{main} in the case $n$ odd and composite.

\subsection{Case $n$ divisible by $4$.} 
We now prove Theorem \ref{main} when $n$ 
is divisible by $4$.

 In \cite{maroti}, Mar\'oti proved that 
$\sigma (G) \sim 2^{n-2}$.
We want to prove that $\omega(G)$ is also 
asymptotic to $2^{n-2}$ in this case, 
so that $\omega(G) \sim \sigma(G)$. 
Since $\omega(G) \leq \sigma(G)$, it is 
enough to find a lower bound for $\omega(G)$ 
which is asymptotic to $2^{n-2}$.

We consider the set
$$
\mathcal{S}= \{ \Delta \subseteq 
\{1,\ldots,n\}, \ 1< |\Delta| < n/2, \ 
|\Delta| \text{ odd } \}.
$$  

Since the number of subsets of 
$\{1,\ldots,n\}$ of even size is equal to 
the number of subsets of odd size (this can be seen 
by expanding the equality $0=(1-1)^n$ with the binomial 
theorem), we have
$$|\mathcal{S}|= \sum_{i=2}^{n/4} 
\binom{n}{2i-1} = 2^{n-2}-n \thicksim 
2^{n-2}.$$

Let $V$ be the set of elements of $A_n$ of 
cycle type $(a,n-a)$ for $a$ odd and 
$1 < a < n/2$. We have $V=\bigcup_{\Delta \in 
\mathcal{S}} C(\Delta)$, where $C(\Delta)$ 
is the set of bicycles with orbits $\Delta$ 
and $\Omega\setminus\Delta$.

As in the case $n$ odd, we define a graph 
$\Gamma$ with now vertex set 
$V(\Gamma)=V$ and whose edge set $E(\Gamma)$ is the family of size 
$2$ subsets $\{ x,y \} \subseteq V$ such 
that $\langle x ,y \rangle \neq A_n$ and 
$x$ and $y$ do not belong to the same 
$C(\Delta)$, for all $\Delta\in\mathcal{S}$.

The sets $C(\Delta)$ determine a 
partition of $V(\Gamma)$ and by Theorem 
\ref{haxell} we are done 
if we can prove that, for all $\Delta \in \mathcal{S}$, 
$$ |C (\Delta)| \geq 2k, $$
where $k$ is the maximum degree of a vertex in $\Gamma$.

In \cite[Theorem 1.5]{guest_primitive_affine} and 
\cite[Theorem 1.1]{guest_primitive} a careful and 
detailed description of the primitive 
permutation groups containing a permutation with at most four cycles is given. From that 
analysis it follows that, when $n$ is divisible by $4$ and sufficiently large 
there are only two cases in 
which a product of two disjoint cycles of odd length in $S_n$ can be contained in a primitive permutation 
group $H \leq S_n$ not containing the 
alternating group $A_n$, and in both cases 
the cycles have lengths $1$ and $n-1$. 
By our choice of $\mathcal{S}$, $|\Delta| 
\neq 1$, therefore, since by the definition of $V$ the 
intransitive subgroups of $A_n$ cannot 
contain subsets of the form $\{x,y\} \in E(\Gamma)$, the 
only maximal subgroups of $A_n$ containing such sets are the imprimitive ones. 

We now evaluate the maximum degree $k$ 
of our graph. Namely, for any 
fixed $x \in C(\Delta)$ we bound the number 
of elements $y \in V 
\setminus C(\Delta)$ such that $\langle x,y 
\rangle \neq A_n$. 

Let $\mathcal{H}^{imp}_x$ be the set of 
maximal imprimitive subgroups of $A_n$ containing 
$x$. The above discussion implies that
$$k \leq \sum_{H \in 
\mathcal{H}^{imp}_x} |H|
$$
Assume that $x \in C(\Delta)$ with 
$|\Delta|=a$, so that $x$ is a product 
of two disjoint cycles of lengths $a$ and 
$n-a$. Moreover assume that $x$ belongs to 
an imprimitive maximal subgroup $W$, say $W 
\simeq S_d \wr S_m$, with $d,m>1$ and $dm=n$.

Then there are two possibilities. 
\begin{itemize}
\item $\Delta$ is the union of some of the 
blocks of $W$. In this case $d \mid a$ and 
$W$ is uniquely determined by $x$, since its 
blocks are exactly the orbits of $x^m$.
\item $m \mid a$ and $\Delta$ (and $\Omega 
\setminus \Delta$) intersects 
each block of $W$ in exactly $a/m$ (resp. 
$(n-a)/m$) elements. 
In this case there are exactly $m$ 
conjugates of $W$ containing $x$: they 
can be obtained by pairing cyclically each 
orbit of $x_1^m$ with an orbit 
of $x_2^m$, where $x_1$ and $x_2$ are 
respectively the restrictions of $x$ to 
$\Delta$ and to 
$\Omega\setminus\Delta$.
\end{itemize}

It follows that 
\begin{eqnarray} \label{impr_4_haxell}
\sum_{H \in \mathcal{H}_x^{imp}} |H|  
 & \leq & \sum_{ r \mid \gcd(a,n)} 
\left( |S_r \wr S_{n/r}|+ r 
\cdot |S_{n/r} \wr S_{r}| \right)\\
 & = & \sum_{r \mid \gcd(a,n)} 
\left( (r!)^{n/r} \left( n/r \right) ! + 
r\cdot r! \cdot ((n/r)!)^r \right). 
\nonumber 
\end{eqnarray}

Since $|C(\Delta)| = \left(|\Delta|-1 
\right) ! \left( n-|\Delta|-1 \right) ! \geq 
(2/n)^2 (n/2)!^2$, inequality (\ref{impr_4_haxell}) together with Lemma \ref{impr_4_n3} gives, for large 
enough $n$,
$$
\frac{k}{|C(\Delta)|} \leq \frac{\sum_{H \in \mathcal{H}_x^{imp}} |H|}{(2/n)^2 (n/2)!^2} \leq  
 \frac{Cn(n/3)!^3}{(2/n)^2 (n/2)!^2} \leq c_2 (2/3)^n n^3
$$
for some constant $c_2$. 
The last inequality can be proved easily using Stirling's 
inequalities (\ref{eq_Stirling}). This proves that 
$k/|C(\Delta)|$ tends to zero as $n \to \infty$, 
hence it is 
smaller than $1/2$ for sufficiently large $n$, which is what 
we wanted to prove.

\section{Proof of Theorem \ref{alt3q}}

The following argument is a slight 
generalization of 
\cite[Section 3]{Swartz}.

Let $G$ be any finite non-cyclic 
group and let $T$ be a finite group 
containing $G$ as a normal subgroup. Let 
$\mathscr{M}$ be a family of maximal 
subgroups of $G$ and let $\Pi$ be a subset 
of $G$. Let $\{M_i\ \vert \ i \in I_{T}\}$ 
be a set of pairwise non-$T$-conjugate 
maximal subgroups of $G$ such that every 
maximal subgroup of $G$ is $T$-conjugate 
to some $M_i$, with $i \in I_T$, and let $\mathscr{M}_i := 
\{ t^{-1}M_it\ :\ t \in T \}$ and 
$\Pi_i := \Pi \cap \bigcup_{M \in 
\mathscr{M}_i} M$, for all $i \in I_{T}$. 
Let $I \subseteq I_{T}$. Suppose that the 
following holds.

\begin{enumerate}
\item $\mathscr{M} = \bigcup_{i \in I} 
\mathscr{M}_i$;
\item $x^{t} \in \Pi$ for all $x \in \Pi$, 
$t \in T$;
\item $\Pi$ is contained in $\bigcup_{M 
\in \mathscr{M}} M$;
\item if $A,B \in \mathscr{M}$ and $A \neq 
B$ then $A \cap B \cap \Pi = \emptyset$;
\item $M \cap \Pi \neq \varnothing$ for 
all $M \in \mathscr{M}$.
\end{enumerate}

Note that this implies in particular that 
$\{\Pi_i\}_{i \in I}$ is a partition of 
$\Pi$. Moreover if $A,B$ are $T$-conjugate 
subgroups then since $\Pi$ and each 
$\Pi_i$ (for $i \in I_{T}$) are unions of 
$T$-conjugacy classes of elements of $T$, 
we have $|A \cap \Pi| = |B \cap \Pi|$ and 
$|A \cap \Pi_i| = |B \cap \Pi_i|$.

For any maximal subgroup $M$ of $G$ 
outside $\mathscr{M}$ define $$d(M) := 
\sum_{i \in I} \frac{|M \cap \Pi_i|}{|M_i 
\cap \Pi_i|}.$$ 
The proof of the following proposition is 
essentially the same as the one in 
\cite[Section 3]{Swartz} but we include it 
for completeness.

\begin{prop}\label{prop:swartz}
Assume the above setting. If $d(M) \leq 1$ for all maximal subgroup 
$M$ of $G$ outside $\mathscr{M}$ then any 
family of proper subgroups of $G$ whose 
union contains $\Pi$ has size at least 
$|\mathscr{M}|$. In other words, 
$\mathscr{M}$ is a minimal covering 
of $\Pi$. Moreover, if $d(M) < 1$ for all 
maximal subgroup $M$ of $G$ outside 
$\mathscr{M}$ then $\mathscr{M}$ is the 
unique minimal covering of $\Pi$.
\end{prop}

\begin{proof}
Let $\mathscr{K}$ be any family of maximal subgroups of $G$ such
that $\bigcup_{K \in \mathscr{K}} K \supseteq \Pi$ and suppose 
$\mathscr{K} \neq \mathscr{M}$. We want to prove that $|\mathscr{M}| \leq |\mathscr{K}|$.
Define 
$$\mathscr{M}' := 
\mathscr{M}-(\mathscr{M} \cap \mathscr{K}), \hspace{1cm}
\mathscr{K}' 
:= \mathscr{K}-(\mathscr{M} \cap \mathscr{K}).
$$ 
For any $i \in I$, let $m_i$ be the number of subgroups from $\mathscr{M}_i$ in $\mathscr{M}'$, and for any $j \in I_T$ let $k_j$ be the number of subgroups from $\mathscr{M}_j$ in $\mathscr{K}'$. 

Observe that since $\mathscr{K}$ covers $\Pi_i$ and $\mathscr{M}$ partitions $\Pi$, the members of $\mathscr{K}'$ must cover the elements of $\Pi_i$ contained in $\bigcup_{M \in \mathscr{M}'} M$. Since $\mathscr{M}$ partitions $\Pi$, the number of such elements is $m_i |M_i \cap \Pi_i|$. Therefore 
$$m_i |M_i \cap \Pi_i| \leq \sum_{j \not \in I} k_j |M_j \cap \Pi_i|.$$
We claim that if $d(M) \leq 1$ for all $M \in \mathscr{K}'$ then $|\mathscr{M}| \leq |\mathscr{K}|$. Indeed, we have
\begin{eqnarray*}
|\mathscr{M}'| & = & \sum_{i \in I} m_i \leq \sum_{i \in I} \sum_{j \not \in I} k_j \frac{|M_j \cap \Pi_i|}{|M_i \cap \Pi_i|} = \\ & = & \sum_{j \not \in I} k_j \sum_{i \in I} \frac{|M_j \cap \Pi_i|}{|M_i \cap \Pi_i|} = \sum_{j \not \in I} k_j d(M_j) \leq \sum_{j \not \in I} k_j = |\mathscr{K}'|.
\end{eqnarray*}
This implies $$|\mathscr{M}| = |\mathscr{M} \cap 
\mathscr{K}|+|\mathscr{M}'| \leq |\mathscr{M} \cap 
\mathscr{K}|+|\mathscr{K}'|=|\mathscr{K}|,$$
and therefore
$\mathscr{M}$ is a covering of $\Pi$ of 
minimal size. 
Moreover, if $d(M) < 1$ for all maximal subgroup $M$ of $G$ 
outside $\mathscr{M}$, then the above argument shows that 
$|\mathscr{M}| < |\mathscr{K}|$ whenever $\mathscr{M} \neq 
\mathscr{K}$, proving that $\mathscr{M}$ is the unique 
covering of $\Pi$ of minimal size.
\end{proof}


From now on let $n\geq 21$
be a positive integer congruent to $3$ 
modulo $18$ and let  $q:=n/3$, $G:=A_n$, 
$T:=S_n$. Note that $q \equiv 1 \pmod 6$. 
We prove Theorem 
\ref{alt3q} by showing (with the use of 
Proposition \ref{prop:swartz}) the 
existence of a 
minimal covering $\mathscr{M}$ for $A_n$  
of size  
$$\sum_{i=1}^{q-2} \binom{n}{i} + 
\frac{1}{6} \frac{n!}{q!^3}.$$

If $n=\sum_{i=1}^t a_i$ and $1 \leq a_1
\leq a_2\leq \ldots \leq a_t$ we denote by 
$(a_1,\ldots,a_t)$ the set of elements of 
$A_n$ whose cycle structure 
consists of $t$ disjoint cycles each of 
length $a_i$, for $i=1,\ldots,t$. 
Note that each $(a_1,\ldots,a_t)$ is 
either empty or an $A_n$-conjugacy class 
or the union of two $A_n$-conjugacy 
classes. The latter case occurs if and 
only if the numbers $a_1,\ldots,a_t$ are 
all odd and pairwise distinct.

Let $\Pi_{-1} =(n)$ be the set of 
all $n$-cycles and for every integer 
$a$ such that $1 \leq a \leq q-2$ define
$$\Pi_a :=
\begin{cases} 
(a,\frac{n-a-1}{2},\frac{n-a+1}{2}) & 
\textrm{if}\ a \equiv 0 \pmod 2 \\
(a,\frac{n-a}{2}-1,\frac{n-a}{2}+1) & 
\textrm{if}\ a \equiv 1 \pmod 2 .
\end{cases}
$$
We define the collection $\mathscr{M}$ of 
$S_n$-conjugacy classes of maximal 
subgroups of $A_n$ as follows.

$\mathscr{M}_{-1}$ is the set of maximal  
imprimitive subgroups of $A_n$ with $3$ 
blocks. Thus the elements of 
$\mathscr{M}_{-1}$ are subgroups 
isomorphic to $(S_q\wr S_3)\cap A_n$.

For every $a$ such that $1\leq a\leq q-2$  
define $\mathscr{M}_a$ to be the set of 
maximal intransitive subgroups 
of $A_n$ which are the stabilizers of a 
set of size $a$.

Finally, let 
$$\Pi := \bigcup_{a=-1,1,\ldots,q-2} \Pi_a 
\quad \textrm{ and } \quad \mathscr{M} 
:= \bigcup_{a=-1,1,\ldots,q-2} 
\mathscr{M}_a.$$ 

In this notation the index set $I$ is 
$\{-1,1,2,\ldots,q-2\}$.

For any $S_n$-conjugacy class 
$\mathscr{M}_j$ of maximal subgroups of 
$A_n$ ($j$ can belong to $I$ or not), let 
$m_j(i)$ be the number of subgroups  
from the $S_n$-class $\mathscr{M}_j$ 
containing a fixed element of $\Pi_i$. The 
number $m_j(i)$ is well-defined because 
each $\Pi_i$ is a $S_n$-conjugacy 
class. Also, as before we denote with 
$I_{S_n}$ an index set for $S_n$-conjugacy class representatives of the maximal 
subgroups of $A_n$.


\begin{lemma} \label{lem_mji}
If $j \in I_{S_n}$ and $M_j \in 
\mathscr{M}_j$ then 
$$
|M_j \cap \Pi_i| = 
\frac{m_j(i) \cdot |N_{S_n}(M_j)| \cdot 
|\Pi_i|}{|S_n|} \leq \frac{m_j(i) \cdot 
|M_j| \cdot |\Pi_i|}{|A_n|}.
$$
Moreover, if $M_j$ is not primitive then 
this inequality is actually an 
equality.
\end{lemma}

\begin{proof}
Consider the bipartite graph with set of 
vertices $\Pi_i \cup \mathscr{M}_j$ and 
where there is an edge between 
$g \in \Pi_i$ 
and $M \in \mathscr{M}_j$ if and only if 
$g \in M$. 
Since $\Pi_i$ is a conjugacy 
class of $S_n$, the family 
$\mathscr{M}_j$ covers $\Pi_i$ if one of 
its members intersects it. 
By assumption the number of edges of this 
graph equals both 
$m_j(i) \cdot |\Pi_i|$ and 
$|S_n:N_{S_n}(M_j)| \cdot |M_j \cap 
\Pi_i|$. We are left to prove that 
$$
|A_n:M_j| \leq |S_n:N_{S_n}(M_j)| 
$$ 
This follows from the fact 
that $M_j$ is self-normalized in $A_n$, 
being a maximal subgroup (and $n \geq 5$), 
and $|S_n:N_{S_n}(M_j)|$ is the number of 
$S_n$-conjugates of $M_j$, while 
$|A_n:M_j|=|A_n:N_{A_n}(M_j)|$ is the 
number of $A_n$-conjugates of $M_j$. 
\end{proof}




\begin{lemma} \label{abc3}
Assume $m$ is a positive integer 
divisible by $3$. 
An element of $S_m$ of cycle type 
$(a,b,c)$, with 
$a,b,c \geq 1$ and $a+b+c=m$, stabilizes 
a partition of 
$\{1,\ldots,m\}$ with $3$ blocks if and 
only if at least one of the following 
holds: 
\begin{enumerate}
\item $a=b=c=m/3$.
\item $3$ divides $\gcd(a,b,c)$;
\item One of $a,b,c$ equals $2m/3$;
\item One of $a,b,c$ equals $m/3$ and the 
other two are even.
\end{enumerate}
\end{lemma}

\begin{proof}
Straightforward.
\end{proof}

We have the following.
\begin{enumerate}
\item $\bigcup_{M \in \mathscr{M}} M = 
A_n$. To see this let 
$g \in A_n$, and let $(a_1,\ldots,a_k)$, 
$1\leq a_1 \leq \ldots 
\leq a_k$, be the cycle type of $g$, with 
$\sum_{i=1}^k a_i = n$. 
Note that, since $g \in A_n$ and $n$ is 
odd, $k$ must be odd. If $a_1 < q-1$ then 
$g$ belongs to a member of 
$\mathscr{M}_{a_1}$. Now assume that 
$a_1 \geq q-1$, so that 
$a_i \geq q-1$ for all $i=1,\ldots,k$. It 
follows that $3q = n = \sum_{i=1}^k a_i 
\geq k(q-1)$, therefore $k \leq 3$ being 
$q>3$ odd. If $k=1$ then $g$ belongs to a 
member of $\mathscr{M}_{-1}$, so now 
assume that $k=3$. Since $q-1 \leq 
a_1 \leq a_2 \leq a_3$, the only 
possibilities for $(a_1,a_2,a_3)$ are 
either $(q-1,q-1,q+2)$ or  
$(q-1,q,q+1)$, therefore 
$g$ belongs to a member of 
$\mathscr{M}_{-1}$ by Lemma 
\ref{abc3} since $q \equiv 1 \pmod 6$ 
(respectively case (2) and case (4)). 
Note that here is the point where we use 
the crucial assumption $n \equiv 3 \pmod{18}$.

\item For every $g \in \Pi$ there exists a 
unique $M \in \mathscr{M}$ such that 
$g \in M$. 
More precisely, if $g \in \Pi_{-1}$ then 
the unique member of $\mathscr{M}$ 
containing $g$ is the unique member of 
$\mathscr{M}_{-1}$ whose blocks are the 
three orbits of $g^3$, and 
if $g\in \Pi_a$, $a\in \{1,\ldots,q-2\}$,  
then the unique member of $\mathscr{M}$ 
containing $g$ is the subgroup in 
$\mathscr{M}_a$ sharing an orbit of size 
$a$ with $g$. This is because no element 
of $\Pi$ which is not an $n$-cycle 
stabilizes a partition with $3$ blocks, 
a fact that can be easily proved by using 
Lemma \ref{abc3}.
\end{enumerate}

From now on let $\mathscr{M}_j$ be a 
$S_n$-class of maximal subgroups of $A_n$ 
not contained in $\mathscr{M}$ (in other 
words we think of $j$ as an index in 
$I_{S_n}\setminus I$) and let $M_j$ be 
any element of $\mathscr{M}_j$. 
We deduce 
from Lemma \ref{lem_mji} that, if 
$i \in I$, then
$$
d(M_j) = \sum_{i \in I} 
\frac{|M_j \cap \Pi_i|}{|M_i \cap \Pi_i|} 
\leq \sum_{i \in I}
\frac{m_j(i)|M_j|}{m_i(i)|M_i|} 
\leq  |M_j| \sum_{i \in I} 
\frac{m_j(i)}{|M_i|}.
$$
Now, if $\mathscr{M}_j$ is a 
$S_n$-class of maximal intransitive 
subgroups of $A_n$ then $m_j(-1)=0$, 
while $m_j(i) \leq 1$ for $1 \leq i \leq 
q-2$ and also $m_j(i)=0$,  except 
for at most $4$ values of $i$. 
This is because, thinking of $j$ as the 
size of an orbit of the members of 
$\mathscr{M}_j$, with $q-1 \leq j < n/2$, 
the possible values of $i$ such that 
$1\leq i\leq q-2$ and $m_j(i)\neq 0$ are 
obtained by solving the equations 
$j=(n-i)/2-1$, $j=(n-i)/2+1$, $j=(n-i-1)/2$ 
and $j=(n-i+1)/2$. 
Note that if $M_j$ is of type 
$(S_{q-1} \times S_{2q+1}) \cap A_n$ 
then $M_j \cap \Pi = \varnothing$, 
implying that $d(M_j)=0$. If this is not 
the case then $|M_j| \leq q!(2q)!$, 
therefore
$$d(M_j) \leq \frac{4 \cdot q! \cdot 
(2q)!}{(q-2)! \cdot (2q+2)!} = 
\frac{4q(q-1)}{(2q+2)(2q+1)} < 1.$$ 

If $\mathscr{M}_j$ is a $S_n$-class 
of transitive subgroups of $A_n$ then 
$m_j(i) \leq n^3$ by Lemma \ref{ncube}. 
Moreover, if $M_j$ is imprimitive then 
$|M_j| \leq (n/5e)^n (5n)^{5/2} e \sqrt{n}$ 
by Lemma \ref{bimpr}, and if $M_j$ is 
primitive then $|M_j| \leq 2^n$ 
by \cite{Maroti_Pr}. 
Since $|M_i| \geq |(S_q \wr S_3) \cap 
A_n| = 3 q!^3 > 3(n/3e)^n$ for every 
$i \in I$ and $|I|<n$, we obtain that
$$d(M_j) \leq 
|M_j| 
\sum_{i \in I} \frac{m_j(i)}{|M_i|} < 
\frac{ 
n^4 (n/5e)^n (5n)^{5/2} e 
\sqrt{n}}{3(n/3e)^n} = 
\frac{
5^{5/2}e}{3} n^7 (3/5)^n < 1,$$
as long as $n\geq 65$.

Finally when $n=21, 39$ or $57$, 
then $q$ is a prime, respectively: $7$, $13$ and $19$. Since $|I|=q-1$ and $m_j(i) \leq n^3$, we can use the bound
$$d(M_j) \leq |M_j| \sum_{i \in I} \frac{m_j(i)}{M_i} \leq \frac{(q-1)n^3|M_j|}{3 \cdot q!^3},$$
which gives the result when $n \in \{39,57\}$ or when $n=21$ and $M_j$ is primitive, by making use of the bound $|M_j| \leq 3!^q \cdot q!$. Here we use the list of primitive subgroups of a given (small) degree, available in
\cite[Table B.2]{DM}.

Now assume $n=21$ and $M=M_j$ is imprimitive, so 
that $M \cong (S_3 \wr S_7) \cap A_{21}$.  
Then the only elements of $\Pi$ that stabilize a partition 
with $7$ blocks are those of type $(21)$ or of type 
$(4,8,9)$. 
Moreover $|M \cap \Pi_{-1}|=|M|/21$
and $|M \cap \Pi_4| = \binom{7}{3} \cdot \frac{3!^4}{9} 
\cdot 3! \cdot 2!^3 = 7! \cdot 48$, while 
$|M_{-1} \cap \Pi_{-1}| = |M_{-1}|/21$
and $|M_{4} \cap \Pi_{4}| = 3! \cdot \binom{17}{8} \cdot 7! 
\cdot 8!$,
hence 
$$d(M) = \frac{3!^7 \cdot 7!}{7!^3 \cdot 3!} + 
\frac{7! \cdot 48}{3! \cdot \binom{17}{8} \cdot 7! 
\cdot 8!} = \frac{315059}{171531360} < 1.$$

\end{document}